\documentclass{amsart}
\usepackage{graphicx}
\usepackage[normalem]{ulem}
\usepackage{amsthm}
 \usepackage{amsfonts}
 \usepackage{amssymb}
 \usepackage{tensor}
 \usepackage{bbm}
 \usepackage{mathrsfs}
 \usepackage{mathtools}
 \numberwithin{equation}{section}
 \usepackage{color}
 \usepackage{hyperref}

\newcommand{\norm}[2]{\ensuremath{\|#2\|_{#1}}}

\newcommand{\Norm}[2]{\ensuremath{\Big\|#2 \Big\|_{#1}}}
\newcommand{\abs}[1]{\ensuremath{|#1|}}

\newcommand{\inner}[2]{\ensuremath{\langle #1,#2 \rangle}_{H}}
\newcommand{\biginner}[2]{\ensuremath{\Big\langle #1,#2 \Big\rangle}_{H}}
\newcommand{\innerv}[2]{\ensuremath{\langle #1,#2 \rangle}_{V}}
\newcommand{\inners}[3]{\ensuremath{\langle #1,#2 \rangle}_{#3}}
\newcommand{\duals}[4]{\ensuremath { \tensor[_{#3}]{\langle#1,#2\rangle 
}{_{#4}}  }}

\newcommand{\dual}[2]{\ensuremath { \tensor[_{V^*}]{ \langle#1,#2\rangle }{_V}  
  }}
\newcommand{\tdual}[2]{\ensuremath { \tensor[_{V}]{ \langle#1,#2\rangle
}{_{V^*}} }}
\newcommand{\bigtdual}[2]{\ensuremath { \tensor[_{V}]{ \Big\langle#1,#2
\Big\rangle }{_{V^*}} }}
\newcommand{\bigdual}[2]{\ensuremath { \tensor[_{V^*}]{ \Big\langle#1,#2
\Big\rangle }{_{V}} }}

\newcommand{\bR}{{\mathbb R}}

\newcommand{\bP}{{\mathbb P}}
\newcommand{\bE}{{\mathbb E}}

\newcommand{\dd}{{\mathrm{d}}}

\newcommand{\ee}{{\mathrm{e}}}

\newcommand{\cY}{\mathcal{Y}}

\newcommand{\cX}{\mathcal{X}}

\newcommand{\cL}{\mathcal{L}}
\newcommand{\cF}{\mathcal{F}}

\newcommand{\cU}{\mathcal{U}}

\newcommand{\cXb}{\ensuremath \mathcal{X}_0^{t,\beta}}
\newcommand{\cYb}{\ensuremath \mathcal{Y}_0^{t,\beta}}
\newcommand{\cXbT}{\ensuremath \mathcal{X}^{\beta}}
\newcommand{\cYbT}{\ensuremath \mathcal{Y}^{\beta}}

\newcommand{\Hb}{\ensuremath \dot{H}^{\beta}}
\newcommand{\AM}{\ensuremath A_{\max}}
\newcommand{\Am}{\ensuremath A_{\min}}

\newcommand{\sC}{\mathscr{C}}

\newcommand{\sL}{\mathscr{L}}

\newcommand{\sS}{\mathscr{S}}

\newcommand{\sT}{\mathscr{T}}
\newcommand{\sB}{\mathscr{B}}
\newcommand{\sG}{\mathscr{G}}
\newcommand{\sW}{\mathscr{W}}

\newcommand{\sF}{\mathscr{F}}

\newtheorem{theorem}{Theorem}
\newtheorem{lemma}[theorem]{Lemma}

\newtheorem{remark}[theorem]{Remark}

\newtheorem{definition}[theorem]{Definition}
\begin{document}

\title{A weak space-time formulation for the
linear stochastic heat equation}

\author[S.~Larsson]{Stig Larsson}

\address{
  Department of Mathematical Sciences,
  Chalmers University of Technology and University of Gothenburg, 
  SE--412 96 Gothenburg,
  Sweden}

\email{stig@chalmers.se}

\author[M.~Molteni]{Matteo Molteni}

\address{
  Department of Mathematical Sciences,
  Chalmers University of Technology and University of Gothenburg,
  SE--412 96 Gothenburg, 
  Sweden}

\email{molteni@chalmers.se}

\begin{abstract}
We apply the well-known Banach-Ne{\v{c}}as-Babu{\v{s}}ka inf-sup theory in a
stochastic setting to introduce a weak space-time formulation of the linear
stochastic heat equation with additive noise. We give sufficient conditions on
the the data and on the
covariance operator associated to the driving Wiener process, in order to have
existence and uniqueness of the solution. We show the relation of the
obtained solution to the mild solution and to the variational solution
of the same problem. The spatial regularity of the solution is also discussed.
Finally, an extension to the case of linear multiplicative noise is presented.
\end{abstract}

\date{\today}

\maketitle
\section{Introduction}\label{intro}
We consider a linear parabolic stochastic evolution problem of the form 
\begin{equation} \label{stocHeat}
\begin{aligned}
&\dd U(t) + A(t)U(t)\,\dd t = f(t)\,\dd t +\Psi(t)\,\dd W(t), &&t \in (0,T], \\
&U(0) = U_0. &&
\end{aligned}
\end{equation}
We assume that $A(t)$ is a random elliptic operator defined within a
Gelfand triple setting as follows. Given separable Hilbert spaces
$V,H$, we consider a Gelfand triple $V\subset H\subset V^*$, where $V$
has a compact and dense embedding into $H$. We denote by
$\inner{\cdot}{\cdot}$ the inner product in $H$ and by
$\dual{\cdot}{\cdot}$ the dual pairing between $V$ and $V^*$ with
$\dual{u}{v}=\inner{u}{v}, \, \forall v\in V$ whenever $u \in
H$. Further, we denote by $\sL(H) = \sL(H;H)$ the space of bounded linear
operators on $H$ and by $\sL_2(H)=\sL_2(H;H)$ the Hilbert-Schmidt operators.

Let $T \in (0,\infty)$ be fixed and let $(\Omega,\Sigma,\bP)$ be a
complete probability space, with normal filtration
$\Sigma = (\Sigma_t)_{t \in [0,T]}$.  We assume that a progressively
measurable map
$ A \colon \Omega \times [0,T] \times V \rightarrow V^*$, coercive and
bounded $ \bP \otimes \dd t$-a.s., is given, with associated bilinear
form $a$ given by $a(\omega,t;u,v) = \dual{A(\omega,t)u}{v}$. We
consider a predictable process with Bochner integrable trajectories
$f \in L^2( \Omega \times (0,T) ;V^*)$ and we assume that
$W = (W(t))_{t \in [0,T]}$ is a $Q$-Wiener process with covariance
operator $Q$, and a predictable operator-valued process $\Psi$ such
that $\Psi{Q}^{\frac{1}{2}} \in L^2(\Omega \times (0,T);\sL_2(H))$.

A typical example would be $V=H^1_0(D)\subset H=L^2(D)$ with a spatial
domain $D$ and a random elliptic operator of the form
$A(\omega,s)u=-\nabla\cdot(a(\omega,s)\nabla u)+b(\omega,s)\cdot\nabla
u+c(\omega,s)u$ with suitable assumptions on the coefficients.

In order to give a meaning to \eqref{stocHeat}, we have to define what
we mean by a solution. In the special case when $A$ is independent of
$t$ and $\omega$ and considered as unbounded operator in $H$, we have
the concepts of {\it weak} and {\it mild solution}, see \cite{DapZab}.
\begin{definition}[Weak and mild solution]
  Let the operator $A$ be possibly unbounded, independent of $\omega$
  and $t$, and defined on a certain domain $D(A)$ dense in
    $H$, i.e., $A \colon D(A)\subset H \rightarrow H$. A weak
  solution to \eqref{stocHeat} is an $H$-valued, predictable
  stochastic process $U(t)$, which is Bochner integrable $\bP$-a.s.\
  and satisfies
\begin{equation}\label{weak}
\begin{aligned}
\inner{U(t)}{v} &= \inner{U_0}{v} - \int_0^t{\inner{U(s)}{A^*v}}\,\dd
s  + \int_0^t{\inner{f(s)}{v}}\,\dd s \\ &\quad + \int_0^t{\inner{\Psi(s)\,\dd
W(s)}{v}},\quad \bP\mbox{\rm{-a.s.}},\,\forall v \in D(A^*),\,t \in [0,T].
\end{aligned}
\end{equation}
In particular $-A$ is the generator a strongly continuous semigroup
$(S(t))_{t\geq 0}$ in $H$ and $\int_0^T{\norm{\sL_2(H)}{S(s)\Psi(s)
Q^{\frac{1}{2}}}^2}\,\dd s < \infty,$ so that the unique weak solution coincides
with the mild solution, given by the formula
\begin{equation}\label{mild}
U(t) =  S(t)U_0 + \int_0^t {S(t-s)f(s)}\,\dd s + \int_0^t{S(t-s)\Psi(s)\,\dd 
W(s)},
\quad t \in [0,T].
\end{equation}
\end{definition} 
Within the semigroup framework
it is possible to prove results about about spatial regularity and temporal
H\"{o}lder-continuity of the solution, by defining spaces of fractional
order, $\dot{H}^{\beta}:=D(A^{\frac{\beta}{2}}) $, and exploiting the semigroup
theory. For example, in the parabolic case, when the semigroup is analytic, it
was shown in \cite{Yan} that if $U_0 \in L^2(\Omega;\dot{H}^{\beta})$,
$f=0$, $\Psi(t)=I$, 
and $ \norm{\sL_2(H)}{A^{\frac{\beta-1}{2}} Q^{\frac{1}{2}}} <\infty$ for some
$\beta \geq 0$, then the mild solution satisfies
\begin{equation*}
\norm{L^2(\Omega;\dot{H}^{\beta})}{U(t)} \leq
C\Big(\norm{L^2(\Omega;\dot{H}^{\beta})}{U_0} + 
\norm{\sL_2(H)}{A^{\frac{\beta-1}{2}}Q^{\frac{1}{2}}} \Big), \quad t\in[0,T].
\end{equation*}

The concept of mild solution presents however the disadvantage of not being
applicable whenever the operator does not generate a semigroup. This fact
provides a good reason to look for more general concepts of solution that do not
rely at all on such a theory. 

For this purpose we recall that, in order to derive the mild solution formula
\eqref{mild}, \cite{DapZab} proceeds from the weak
formulation \eqref{weak} with time-independent deterministic test
functions, to a weak formulation with time-dependent deterministic
test functions, c.f.~Lemma~\ref{lemma:versions} below:
\begin{equation*}
\begin{aligned}
\inner{U(t)}{v(t)} &= \inner{U_0}{v(0)} +
\int_0^t{\inner{U(s)}{\dot{v}(s) - A^*v(s)}}\,\dd s \\ & \quad +
\int_0^t{\inner{f(s)}{v(s)}}\,\dd s + \int_0^t{\inner{\Psi(s)\,\dd W(s)}{v(s)}}.
\end{aligned}
\end{equation*}
This suggests the possibility of using a weak space-time formulation, which
would be to find a pair $(U_1,U_2)$ such that 
\begin{equation*}
\begin{aligned}
&\int_0^t{\inner{U_1(s)}{-\dot{v}(s) + A^*v(s)}}\,\dd s +  \inner{U_2}{v(t)}\\ 
&\qquad= \inner{U_0}{v(0)} + \int_0^t{\inner{f(s)}{v(s)}}\,\dd s +
\int_0^t{\inner{\Psi(s)\,\dd W(s)}{v(s)}},
\end{aligned}
\end{equation*}
for all $v$ in a suitable class of test functions.

With a proper choice of function spaces, the well-posedness of this
problem in the deterministic setting is obtained within the
Banach-Ne{\v{c}}as-Babu{\v{s}}ka inf-sup theory, see Section
\ref{preliminary} below.
In Section \ref{varformstoch} we extend this
to the stochastic evolution problem \eqref{stocHeat}. The equation is
solved $\omega$-wise and the inf-sup theory allows us to prove that a
solution exists, is unique, and satisfies a bound that is expressed in
terms of the data $U_0$, $f$, $\Psi$, and $W$, $\bP$-a.s. By taking the
expectation of this, we achieve a standard estimate for the norm of
the solution in the space $L^2( \Omega \times (0,T) ;V)\cap
L^2(\Omega;\sC([0,T];H))$, which is consistent with standard estimates
presented, for example, in \cite[Chapt.~5]{Chow}. In particular, under
suitable assumptions, our solution coincides with the mild solution.
In Section \ref{Regularity}, we briefly discuss the spatial regularity
under such assumptions.

A more general solution concept is the {\it variational
  solution}, for which a comprehensive theory can be found, for
example, in \cite[Chapt.~4]{PrevotRockner}. This theory applies to
more general quasilinear equations, but we present it here for our
linear equation.  

\begin{definition}[Variational solution]\label{varsol}
  Assume that $\Psi$ and $Q$ are as before, that is to say,
  $\Psi Q^{\frac12}\in L^2( \Omega \times (0,T) ;\sL_2(H))$. A
  continuous $H$-valued $\Sigma$-adapted process
  $(U(t))_{t \in [0,T]}$ is called a variational solution to
  \eqref{stocHeat}, if for its $ \bP \otimes \dd t $ equivalence class
  $\hat{U}$ we have
  $\hat{U} \in L^2(\Omega \times (0,T) , \bP \otimes \dd t ;V)$ and,
  for any $t \in [0,T]$,
\begin{equation*} 
U(t) = U_0 - \int_0^t{A(s)\bar{U}(s)}\,\dd s+ \int_0^t{f(s)}\,\dd s+
\int_0^t{\Psi(s)\,\dd W(s)},\quad \text{$\bP$-a.s.,}
\end{equation*}
where $\bar{U}$ is any $V$-valued progressively
measurable $ \bP \otimes \dd t$ version of $\hat{U}$.
\end{definition}

We show in Lemma~\ref{lemma:versions} that our solution coincides with
such a solution, in particular, that our $U_1$ and $U_2$ play the
roles of the $\bar{U}$ and $U$, respectively.  

Finally, the norm bound that we obtain for the solution operator of
the linear problem with additive noise allows us to use a standard
fixed point technique and extend our theory to the case of
multiplicative noise. In Section \ref{linmultnoise} we present this in
the case of linear multiplicative noise.  This approach extends to
semilinear equations under appropriate global Lipschitz assumptions.

We want to remark that despite the fact that the concept of solution that we 
present is essentially no more general than the ones already known, it presents 
two advantages. It allows in fact the development of a theory for existence 
and uniqueness that is relatively easier than others and it states the problem 
in a way that can naturally be used for Petrov-Galerkin approximation of the 
problem. For this second reason our work can be seen as the potential 
starting point for future works dealing with numerical solutions for 
\eqref{stocHeat}.

\section{Preliminaries} \label{preliminary}
\subsection{The inf-sup theory}\label{preliminary1}
We recall the Banach-Ne{\v{c}}as-Babu{\v{s}}ka (BNB) theorem, see
\cite{BabuskaAziz,Ern}, for example.  Let $V$ and $W$ be Banach
spaces, $W$ reflexive, and consider a bounded bilinear form $\sB
\colon W \times V \rightarrow \bR $, with
\begin{equation}
\tag{BDD} \label{BDD}
C_B : =  \sup_{0 \neq w \in W}\sup_{0 \neq v \in
V}\frac{\sB(w,v)}{\norm{V}{v}\norm{W}{w}} < \infty,\\
\end{equation}
and the associated bounded linear operator $B \colon W \rightarrow
V^*$, i.e., $B\in\sL(W,V^*)$, defined by $$ \duals{Bw}{v}{V^*}{V} :=
\sB(w,v) , \quad \forall w\in W, \forall v \in V.$$ The operator $B$
is boundedly invertible if and only if the following conditions are
satisfied:
\begin{align}
\tag{BNB1} \label{BNB1A}
&c_B :=  \inf_{0 \neq w \in W}\sup_{0 \neq v \in
V}\frac{\sB(w,v)}{\norm{V}{v}\norm{W}{w}} >0,\\
\tag{BNB2} \label{BNB2A}
&\forall 0 \neq v \in V, \quad \sup_{0 \neq w \in W}\sB(w,v) >0.
\end{align} 

The constant $c_B$ is called the inf-sup constant and, whenever both $V$ and $W$
are reflexive and \eqref{BNB1A} holds, we have the identity
\begin{align}\label{interchangespaces}
\inf_{0 \neq w \in W}\sup_{0 \neq v \in
V}\frac{\sB(w,v)}{\norm{V}{v}\norm{W}{w}} = \inf_{0 \neq v \in V}\sup_{0 \neq w
\in W}\frac{\sB(w,v)}{\norm{V}{v}\norm{W}{w}},
\end{align}
which allows to swap the spaces where the infimum and the supremum are taken.

An immediate consequence of this is that the variational problem:
\begin{itemize}
\item[]
\quad given $F \in V^*$, find $w \in W \colon \sB(w,v) = F(v), \quad \forall 
v\in V,$
\end{itemize}
i.e., solve $Bw=F$ in $V^*$, and its adjoint: 
\begin{itemize}
\item[]
\quad given $G \in W^*$, find $v \in V\colon \sB(w,v) = G(w), \quad \forall w\in 
W,$
\end{itemize}
i.e., solve $B^*v=G$ in $W^*$, are well-posed whenever \eqref{BDD},
\eqref{BNB1A} and \eqref{BNB2A} hold. In particular, the
well-posedness of the former is equivalent to the well-posedness of
the latter and the respective solutions satisfy
\begin{align*}
\norm{W}{w} \leq \frac{1}{c_B}\norm{V^*}{F}, \quad \norm{V}{v} \leq
\frac{1}{c_B}\norm{W^*}{G}.
\end{align*}

\subsection{The inf-sup theory applied to an abstract parabolic
problem}\label{preliminary2}
In recent years there has been a renewed interest for the tools
presented above in order to deal with the linear heat equation
starting from an abstract parabolic equation given in the Gelfand
triple framework (see, for example,
\cite{BabuskaJanik,CheginiStev,SchwabSte,SchwabSuli,Fra,UrbanPatera}). Assume
indeed that Hilbert spaces $V, H$ are given, forming a Gelfand triple
$ V \subset H \subset V^*$ with bilinear forms
\begin{align*} 
a(t\,;\cdot,\cdot) \colon V \times V \rightarrow \bR,\quad t\in [0,T], 
\end{align*}
satisfying the following conditions for some positive numbers $\Am, \AM$:
\begin{equation}\label{eq:bounds_on_A}
\begin{aligned}
& \abs{a(t;u,v)} \leq \AM \norm{V}{u} \norm{V}{v},\quad && t
\in [0,T],\ u,v\in V, \\
& a(t;v,v) \geq \Am \norm{V}{v}^2, \quad && t \in [0,T],
\ v\in V.
\end{aligned}
\end{equation}
For every $t \in [0,T]$, let $ A(t)$ be the bounded linear operator from $V$ to
$V^*$ associated with the bilinear form, i.e., $A(t) \in \cL(V,V^*)$ and
\begin{align*}
\dual{A(t)u}{v} = a(t;u,v) = \tdual{u}{A^*(t)v}.
\end{align*}

Consider now the problem
\begin{equation} \label{equationstrongform}
\begin{aligned}
&\dot{u}(t) + A(t)u(t) = f(t) &&\text{in $V^*$, $ t \in (0,T)$}, \\
&u(0) = u_0 &&\text{in $H$},
\end{aligned}
\end{equation}
where $\dot{u}(t)$ denotes the derivative of $u$ with respect to $t$, i.e.,
$\dot{u}(t):=\frac{\dd u}{\dd t}$. Define the Lebesgue-Bochner spaces
\begin{align*}
\cY &= L^2((0,T);V),\\
\cX &= L^2((0,T);V) \cap H^1((0,T);V^*),
\end{align*}
normed by
\begin{align*}
\norm{\cY}{y}^2 &= \norm{L^2((0,T);V)}{y}^2 = \int_0^{T}\norm{V}{y(t)}^2\,\dd
t,\\
\norm{\cX}{x}^2 &= \norm{L^2((0,T);V)}{x}^2 + \norm{L^2((0,T);V^*)}{\dot{x}}^2 
+ \norm{H}{x(0)}^2 + \norm{H}{x(T)}^2.
\end{align*}
The trace theorem for Bochner-Lebesgue spaces 
  (\cite[Theorem 1, Chapter 
  XVIII.1]{DauLio}), says that $\cX$ is densely embedded in
  $\sC([0,T];H)$, so that $x(0),x(T)\in H$ are defined. Due to the inclusion of 
the boundary terms in the 
  norm, the embedding constant does not depend on time and is uniform in the 
  choice of $V$.  In fact, whenever $x,y \in \cX$, integration by parts 
is possible: 
\begin{align}  \label{intbyparts}
\int_0^t{\Big(\dual{\dot{x}}{y} + \tdual{x}{\dot{y}}\Big)}\,\dd s =
\inner{x(t)}{y(t)} - \inner{x(0)}{y(0)}.
\end{align}
Hence, for arbitrary 
  $x\in\cX$ and $t \in [0,T]$,
\begin{align*}
\norm{H}{x(t)}^2 = \norm{H}{x(0)}^2 + 2 \int_0^t{ \tdual{x(s)}{\dot{x}(s)} 
}\,\dd s.
\end{align*}
This implies  
\begin{align*}
\sup_{t\in[0,T]}\norm{H}{x(t)}^2 \leq \norm{H}{x(0)}^2 
+\norm{L^2((0,T);V)}{x}^2 
+ \norm{L^2((0,T);V^*)}{\dot{x}}^2 \leq \norm{\cX}{x}^2,
\end{align*}
which leads to the following estimate for the embedding constant:
\begin{equation}\label{eq:embedding_constant}
\sup_{0 \neq x \in \cX} \frac{\norm{\sC([0,T];H)}{x}}{\norm{\cX}{x}} \leq 1.
\end{equation}
The reader can refer to \cite[Chapter XVIII]{DauLio} for a comprehensive 
presentation of these
spaces.

A possible approach to solving the differential problem
\eqref{equationstrongform} is presented for example in \cite{SchwabSte} and it
consists in integrating in time the dual pairing between the equation and a test
function $y_1 \in \cY$ and taking the inner product between the initial
condition and another test vector $y_2 \in H$, thus obtaining the following two
equations:
\begin{equation*}
\begin{aligned}
&\int_0^T{ \Big( \dual{ \dot{u} (t) }{ y_1(t) } + a(t; u(t),y_1(t) ) \Big)}\,\dd
t = \int_0^T {\dual{ f(t) }{ y_1(t) } }\,\dd t, \\
&\inner{u(0)}{y_2} = \inner{u_0}{y_2}.
\end{aligned}
\end{equation*}
Adding the equations and defining $\cY_H:=\cY \times H$, Hilbert space normed by
its product
norm, gives the variational problem
\begin{equation}\label{primalformulation}
u \in \cX :  \sB(u,y) = \sF(y), \quad \forall y = (y_1,y_2) \in \cY_H,
\end{equation}
where the following bilinear and linear forms are used
\begin{align*}
&\sB \colon \cX \times \cY_H \rightarrow \bR,\\
&\sB(x,y) := \int_0^T{ \Big( \dual{\dot{x}(t)}{y_1(t)} +a(t;x(t),y_1(t))
\Big)}\,\dd t + \inner{x(0)}{y_2},\\
&\sF \colon \cY_H \rightarrow \bR,\\
&\sF(y) := \int_0^T {\dual{f(t)}{y_1(t)}}\,\dd t + \inner{u_0}{y_2}.
\end{align*}
We call this the {\it first} space-time variational formulation of
\eqref{equationstrongform}.
in

Consider now the backward adjoint problem to \eqref{equationstrongform}:
\begin{equation} \label{adjequationstrongform}
\begin{aligned}
&-\dot{v}(t) + A^*(t)v(t) = g(t) &&\mbox{in $V^*$, $t \in (0,T)$}, \\
&v(T) = \xi &&\mbox{in $H$},
\end{aligned}
\end{equation}
whose first space-time variational formulation is given by
\begin{equation}\label{firstbackward}
v \in \cX : \sB^*(y,v) = \sG(y), \quad \forall y \in \cY_H.
\end{equation}
Here the bilinear form is given by
\begin{equation*}
\begin{aligned}
&\sB^* \colon \cY_H \times \cX \rightarrow \bR,\\
&\sB^*(y,x) := \int_0^T{\Big( \tdual{y_1(t)}{-\dot{x}(t)} +a(t;y_1(t),x(t))
\Big)}\,\dd t + \inner{y_2}{x(T)},
\end{aligned}
\end{equation*}
and the load functional by
\begin{equation*}
\begin{aligned}
&\sG \colon \cY_H \rightarrow \bR,\\
&\sG(y) := \int_0^T {\duals{y_1(t)}{g(t)}{V}{V^*} }\,\dd t + \inner{y_2}{\xi}.
\end{aligned}
\end{equation*}

Note that $\cX \subseteq \cY_H$ via the embedding $y_1(t)=x(t)$,
$y_2=x(0)$.  By considering the restriction of the load functional
$\sF$ to $\cX \subseteq \cY_H$,
\begin{align*} 
&\sF \colon \cX \rightarrow \bR,\\
&\sF(x) := \int_0^T {\dual{f(t)}{x(t)}}\,\dd t 
+
\inner{u_0}{x(0)},
\end{align*}
and by interchanging the roles of trial and test spaces, the {\it
  second} (or {\it weak}) space-time formulation of the original
problem \eqref{equationstrongform} is obtained:
\begin{equation} \label{secondspacetime}
u=(u_1,u_2) \in \cY_H: \sB^*(u,x) = \sF(x), \quad \forall x \in \cX.
\end{equation}

The first and the second formulations are related and the
well-posedness of the former is equivalent to the well-posedness of
the latter.  More precisely, it holds that (by a suitable modification
of the proofs in \cite{SchwabSte,Fra})
\begin{equation}\label{infsupconstants}
\begin{aligned}
&C_B := \sup_{0 \neq x\in\cX} \sup_{0\neq y\in\cY_H}{
\frac{\sB^*(y,x)}{\norm{\cX}{x}\norm{\cY_H}{y}}} \leq
\sqrt{2\max\{1,\AM^2\}}, \\
&c_B := \inf_{0\neq x\in\cX} \sup_{0\neq y\in\cY_H}{
\frac{\sB^*(y,x)}{\norm{\cX}{x}\norm{\cY_H}{y}}} \geq 
\frac{\min\{\Am,\AM^{-1},\Am\,\AM^{-1}\}}{{2}},
\end{aligned}
\end{equation}
and, for any $ y \in \cY_H$,  
\begin{equation*}
\sup_{0\neq x \in \cX} \sB^*(y,x) 
\geq \min{\{ 1, \Am\}}\norm{\cY_H}{y}^2. 
\end{equation*}
This shows that the operator ${B^*\in\cL(\cX,\cY_{H}^*)} $, associated
with the bilinear form $ \sB^*(\cdot,\cdot)$ via $ \sB^*(y,x) =
\duals{y}{B^*x}{\cY_H}{\cY_H^*}$ is boundedly invertible. This, in
turn, implies that the operator $B \in \cL(\cY_{H},\cX^*)$ associated
with $ \sB^*(\cdot,\cdot)$ via $ \sB^*(y,x) =
\duals{By}{x}{\cX^*}{\cX}$ is also boundedly invertible, with the same
inf-sup constant, see \eqref{interchangespaces}. Moreover, for $ f \in
L^2((0,T);V^*)$ and $ u_0 \in H$, we have $\sF \in \cX^*$. Hence,
\eqref{secondspacetime} is well-posed.

If a solution of \eqref{secondspacetime} has the additional regularity
$u_1 \in \cX$, then an integration by parts \eqref{intbyparts} shows
that $u_1$ is a solution of the first problem
\eqref{primalformulation} and that $u_2 = u_1(T)$.  This is the case
when $ f \in L^2((0,T);V^*)$, as is easily seen.  In this case the
second component of the solution, $u_2$, is a continuous $H$-valued
version of $u_1$, evaluated at time $t=T$. Therefore, $u_2$ is
redundant and in other works, e.g., \cite{CheginiStev} and
  \cite{SchwabSuli}, the weak space-time formulation is
\begin{equation*} 
u \in \cY: \sB^*(u,x) = \sF(x), \quad \forall x \in
\cX_{0,\{T\}}:=\{x\in\cX:x(T)=0\}.
\end{equation*}
Of course, more general functionals $\cF$ may be considered for which
$u_1\not\in\cX$, e.g.,
$\cF(x)=\int_0^T\dual{\dot{x}(t)}{g(t)}\,\dd t$.  As another example,
in the next section we add a noise term to $\cF$.  Then we find it
useful to keep $u_2$.

\section{A weak space-time formulation of the stochastic
problem}\label{varformstoch}
\subsection{Existence and uniqueness} \label{subseq:exist}
In order to introduce the weak space-time formulation for the equation
\eqref{stocHeat} we will follow the idea outlined in Subsection
\ref{preliminary2}.  We consider spaces $\cX$ and $\cY$ restricted
to a time interval $[0,t]$, for fixed but arbitrary $t \in [0,T]$, endowed with
their respective natural norms.  We denote these spaces 
\begin{equation*}
\cY_0^t = L^2((0,t);V), \quad
\cX_0^t = L^2((0,t);V) \cap H^1((0,t);V^*),
\end{equation*}
normed by
\begin{align*}
\norm{\cY_0^t}{y}^2 &= \norm{L^2((0,t);V)}{y}^2 , \\
\norm{\cX_0^t}{x}^2 &= \norm{L^2((0,t);V)}{x}^2 + 
\norm{L^2((0,t);V^*)}{\dot{x}}^2 
+ \norm{H}{x(0)}^2 + \norm{H}{x(t)}^2, 
\end{align*}
with the convention that $\cX = \cX_0^T$ and $\cY = \cY_0^T$.
The reason for introducing the parameter $t\in[0,T]$ is
  that we want to display the time dependence of $u_2$, so that we can
  take the supremum with respect to $t$ and obtain norms and spaces
  consistent with the ones used in \cite{PrevotRockner}.

We assume that the family of operators $A(\omega,s)$ is as in
Section~\ref{intro}, i.e., that its bilinear forms satisfy the
following conditions for some positive numbers $\Am, \AM$:
\begin{equation*}
\begin{aligned}
& \abs{a(\omega,s;u,v)} \leq \AM \norm{V}{u} \norm{V}{v},\quad
&& (\omega,s) \in \Omega \times [0,T], && \, u,v\in V, \\
& a(\omega,s;v,v) \geq \Am \norm{V}{v}^2, \quad &&(\omega,s) \in \Omega \times 
[0,T], &&\, v\in V.
\end{aligned}
\end{equation*}

We introduce a family of problems parametrized by $(\omega,t)$, defined
by the bilinear forms
\begin{equation*}
\begin{aligned}
&\sB_{\omega,t}^* \colon (\cY_0^t \times H) \times \cX_0^t \rightarrow \bR,\\
&\sB_{\omega,t}^*(y,x) := \int_0^t
\Big(\dual{y_1(s)}{-\dot{x}(s)}
+a(\omega,s;y_1(s),x(s)) \Big)\,\dd s 
+
\inner{y_2}{x(t)},
\end{aligned}
\end{equation*}
and the load functionals 
\begin{align*}
\sF_{\omega,t} \colon \cX_0^t  \rightarrow \bR,\quad 
\sW_{\omega,t} \colon \cX_0^t  \rightarrow \bR,
\end{align*}
where
\begin{align*}
&\sF_{\omega,t}(x) = \int_0^t {\dual{f(\omega,s)}{x(s)}}\,\dd s +
\inner{U_0(\omega)}{x(0)},\\
&\sW_{\omega,t}(x) = \Big(\int_0^t{\inner{\Psi(s)\,\dd 
W(s)}{x(s)}}\Big)(\omega) .
\end{align*}
The weak space-time formulation reads, for almost every $(\omega,t)
\in \Omega \times [0,T]$:
\begin{equation} \label{secondspacetimestoch}
U_{\omega,t} \in \cY_0^t \times H : \sB_{\omega,t}^*(U_{\omega,t},x) 
= \sF_{\omega,t}(x) + \sW_{\omega,t}(x), 
\quad \forall x \in \cX_0^t.
\end{equation}

Since our assumption on $a(\omega,s;\cdot,\cdot)$ is uniform with
respect to $\omega,s$ with constants $\Am,\AM$, we conclude that
the bilinear forms $\sB_{\omega,t}^*$ satisfy the inf-sup conditions
uniformly in $\omega,t$ with the same constants $C_B, c_B$ as in
\eqref{infsupconstants}.  This means that, for almost every $(\omega,t)
\in \Omega \times 
[0,T]$, the operator $B_{\omega,t} \in \cL(\cY_0^t
\times {H},(\cX_0^t)^*)$ associated to $\sB_{\omega,t}^*(\cdot,\cdot)$
via $ \sB_{\omega,t}^*(y,x) =
\duals{B_{\omega,t}y}{x}{(\cX_0^t)^*}{\cX_0^t}$ is boundedly
invertible.  Moreover, the norm of its inverse 
$B_{\omega,t}^{-1}$ is bounded by $c_B^{-1}$, uniformly in $\omega,t$.

Focusing now on the right-hand side, we assume that
$f(\omega,\cdot) \in L^2((0,T);V^*)$ and that $U_0(\omega) \in
H$. Then, for $ x \in \cX_0^t$, it holds that
\begin{align*}
\abs{\sF_{\omega,t}(x)} &= \Big| {\int_0^t{\dual{f(\omega,s)}{x(s)}\,\dd s +
\inner{U_0(\omega)}{x(0)}}} \Big|\\
&\leq {\Big(\int_0^t{\norm{V^*}{f(\omega,s)}^2}\,\dd
s\Big)^{\frac{1}{2}}\Big(\int_0^t{\norm{V}{x(s)}}^2\,\dd
s\Big)^{\frac{1}{2}}} +
\norm{H}{U_0(\omega)}\norm{H}{x(0)}\\
&\lesssim \big(\norm{L^2((0,t);V^*)}{f(\omega,\cdot)} +
\norm{H}{U_0(\omega)}\big) 
\norm{\cX_0^t}{x},
\end{align*}
showing that $\sF_{\omega,t} \in (\cX_0^t)^*$ with 
\begin{align}
  \label{eq:boundF}
\norm{(\cX_0^t)^*}{\sF_{\omega,t} } \lesssim 
\norm{L^2((0,T);V^*)}{f(\omega,\cdot)} + \norm{H}{U_0(\omega)} 
\end{align}
for almost every
$(\omega,t) \in \Omega \times 
[0,T]$.  Hence, by monotonicity
in $t$, it follows that
\begin{equation}\label{boundrhsmeansq}
\bE\Big[ \sup_{t\in[0,T]}\norm{(\cX_0^t)^*}{\sF_{\cdot,t} } \Big] \lesssim 
\bE\Big[ \norm{L^2((0,T);V^*)}{f} + \norm{H}{U_0}\Big]. 
\end{equation}

The next step is provided by the following lemma, which shows that
${\sW}_{\omega,t} \in (\cX_0^t)^*$ with an estimate similar to the one
in \eqref{boundrhsmeansq}.  In order to prove this, we let
$A_0 \in\cL(V, V^*)$ be the operator associated with the bilinear form
$a_0(\cdot,\cdot)=\innerv{\cdot}{\cdot}$.  Then $A_0$ does not depend
on $(\omega,t)$ and satisfies the boundedness and coercivity
\eqref{eq:bounds_on_A} with constants $\Am=\Am=1$.  Then $-A_0$ is
self-adjoint and the generator of an analytic semigroup
$(S_0(t))_{t\geq 0} $, which is also self-adjoint,
$(S_0(t))^*=S_0(t)$.  Due to the compact embedding $V\subset H$ and
the spectral theorem there is an orthonormal eigenbasis for $A_0$ in
$H$.  We denote the eigenpairs by $(\lambda_j,\phi_j)$,
$j=1,\dots,\infty$.

In the generic example, where $V=H^1_0(D)\subset H=L^2(D)$ with
$\innerv{u}{v}=\inner{\nabla u}{\nabla v}$ and elliptic
operator of the form $A(\omega,s)u=-\nabla
\cdot(a(\omega,s)\nabla u)+b(\omega,s)\cdot\nabla u+c(\omega,s)u$, we
would have $A_0=-\Delta$, the Dirichlet Laplacian.

\begin{lemma}\label{lemmastochint}
If  $\Psi{Q}^{\frac{1}{2}} \in L^2(\Omega \times (0,T) ;\sL_2(H))$, then there
exists a process $K\in L^2(\Omega;\sC([0,T];\bR))$
such that, for almost every $(\omega,t) \in \Omega \times 
[0,T]$, 
\begin{equation}\label{boundstochintomegawise1}
\norm{(\cX_0^t)^*}{\sW_{\omega,t}} \lesssim K(\omega,t)
\end{equation}
and 
\begin{equation}\label{boundstochintmeansq1}
\bE\Big[ \sup_{t\in[0,T]} K(\cdot,t)^2\Big]
\lesssim 
\bE\Big[ \int_0^T{\norm{\sL_2(H)}{\Psi(t)Q^{\frac12}}^2}\,\dd t\Big].  
\end{equation}
Hence, 
${\sW}_{\omega,t}
\in (\cX_0^t)^*$ for almost every $(\omega,t) \in \Omega \times 
[0,T]$
and 
\begin{align}\label{boundstochintmeansq2}
\bE\Big[ \sup_{t\in[0,T]} \norm{(\cX_0^t)^*}{{\sW}_{\cdot,t}}^2 
\Big]\lesssim \bE\Big[
\int_0^T{\norm{\sL_2(H)}{\Psi(t)Q^{\frac12}}^2}\,\dd t\Big]. 
\end{align}
The constant hidden in $\lesssim$ depends only on numerical
factors.
\end{lemma}

\begin{proof}
  We consider the adjoint problem \eqref{adjequationstrongform} on
  $[0,t]$, with $A^*(\cdot)$ replaced by $A^*_{0}=A_0$.  Problem
  \eqref{adjequationstrongform} is well-posed, i.e., the operator
  ${B_0^*\colon \cX_0^t \rightarrow (\cY_0^t \times H)^*}$ associated
  with the bilinear form in \eqref{firstbackward} is a bijection.
  From the theory of operator semigroups we recall that the solution
  operator $(B_0^{*})^{-1}$ can be represented by the mild solution
  formula,
\begin{equation*}
v(s)=\big( (B_0^{*})^{-1} (g,\xi)  \big)(s)
= \int_s^t{ S_0(r-s)g(r)}\,\dd r 
+ S_0(t-s)\xi, \quad s\in[0,t]. 
\end{equation*}
In order to see this we must show that the mild formula belongs to
$\cX_0^t$, whenever the data $(g,\xi)$ belong to $ (\cY_0^t \times
H)^*=L^2((0,t);V^*)\times H$; more precisely,  
\begin{align*}
  \norm{\cX_0^t}{v}\lesssim\norm{L^2((0,t);V^*)}{g} + \norm{H}{\xi}.  
\end{align*}
This is proved by means of an eigenbasis expansion of the operator
$S_0(t)$ and Parseval's identity.  For example, one term is
\begin{align*}
  &\Norm{L^2((0,t);V)}{\int_{\cdot}^t{ S_0(r-\cdot)g(r)}\,\dd r }^2
=
\int_0^t \Norm{V}{\int_s^t{ S_0(r-s)g(r)}\,\dd r }^2\,\dd s 
\\ & \quad
\lesssim
\int_0^t \Norm{H}{A_0^{\frac12}\int_s^t{ S_0(r-s)g(r)}\,\dd r }^2\,\dd s 
\\ &  \quad
=
\int_0^t \sum_{j=1}^{\infty} 
\Big(
\int_s^t 
\lambda_j^{\frac12}\ee^{-\lambda_j(r-s)}\inner{g(r)}{\phi_j}
\,\dd r 
\Big)^2 
\,\dd s 
\\ & \quad
\le 
\int_0^t \sum_{j=1}^{\infty} 
\Big(
\int_s^t 
\big(\lambda_j^{\frac12}\ee^{-\frac12\lambda_j(r-s)}\big)^2
\,\dd r 
\int_s^t 
\big(\ee^{-\frac12\lambda_j(r-s)}
\inner{g(r)}{\phi_j}\big)^2 
\,\dd r 
\Big)
\,\dd s 
\\ &  \quad
\le 
\sum_{j=1}^{\infty} 
\int_0^t 
\int_s^t 
\big(\ee^{-\frac12\lambda_j(r-s)}
\inner{g(r)}{\phi_j}\big)^2 
\,\dd r 
\,\dd s 
\\ &  \quad
=
\sum_{j=1}^{\infty} 
\int_0^t 
\int_0^r 
\ee^{-\lambda_j(r-s)}
\,\dd s \, 
\inner{g(r)}{\phi_j}^2 
\,\dd r 
\\ &  \quad
=
\sum_{j=1}^{\infty} 
\int_0^t 
\lambda_j^{-1}
\inner{g(r)}{\phi_j}^2 
\,\dd r 
\lesssim 
\norm{L^2((0,t);V^*)}{g}^2.
\end{align*}
The remaining terms in $\norm{\cX_0^t}{v}$ are treated similarly. 

We can hence write any $x\in
\cX_0^t$ as $x=(B_0^{*})^{-1}B_0^*x$, which is represented as:
\begin{equation*} 
x(s) = ((B_0^{*})^{-1}B_0^*x)(s) 
= \int_s^t{ S_0(r-s)(-\dot{x}(r)+ A_0 x(r))}\,\dd r 
+ S_0(t-s)x(t).  
\end{equation*} 
We insert this expression into the weak stochastic integral to get
\begin{align*}
&\int_0^t \inner{\Psi(s)\,\dd W(s)}{x(s)}\\
&\qquad =\int_0^t {\biginner{\Psi(s)\,\dd W(s)}{\int_s^t{
S_0(r-s)(-\dot{x}(r)+ A_0 x(r)) }\,\dd r }}\\
&\quad\qquad  + \int_0^t {\inner{\Psi(s)\,\dd
W(s)}{S_0(t-s)x(t)}}\\
&\qquad =\int_0^t{ \biginner{  \int_0^r{ S_0(r-s) \Psi(s)\,\dd W(s)}}{
(-\dot{x}(r)+ A_0 x(r)) }}\,\dd r \\
&\quad\qquad  + \biginner{ \int_0^t { S_0(t-s) \Psi(s)\,\dd
W(s)}}{x(t)}.
\end{align*}
Here we used the stochastic Fubini theorem and
$(S_0(t))^*=S_0(t)$. It follows that
\begin{align*}
&\Big|{\int_0^t \inner{\Psi(s)\,\dd W(s)}{x(s)} }\Big|\\
&\quad \leq \Big(\int_0^t{\Norm{V}{ \int_0^r{ S_0(r-s) \Psi(s)\,\dd W(s)}}^2} 
\,\dd r\Big)^{\frac12} 
\Big( \int_0^t{\norm{V^*}{-\dot{x}(r)+ A_0 x(r) }^2} 
\,\dd r \Big)^{\frac12} \\
&\qquad  + \Norm{H}{ \int_0^t { S_0(t-s) \Psi(s)\,\dd W(s)}} 
\norm{H}{x(t)}\\
&\quad 
\lesssim 
\Big( \int_0^t{\Norm{V}{
\int_0^r{
S_0(r-s) \Psi(s)\,\dd W(s)}}^2} 
\,\dd r
\\ &\qquad 
+ \Norm{H}{ \int_0^t { S_0(t-s) \Psi(s)\,\dd W(s)}}^2\Big)^{\frac12} 
\norm{\cX_0^t}{x},
\end{align*}
where the constant hidden in $\lesssim$ depends only on numerical 
factors.

This implies \eqref{boundstochintomegawise1} with 
\begin{equation*}
K(\cdot,t) 
:= \Big( \int_0^t{\Norm{V}{
\int_0^r{
S_0(r-s) \Psi(s)\,\dd W(s)}}^2} 
\,\dd r
+ \Norm{H}{ \int_0^t { S_0(t-s) \Psi(s)\,\dd W(s)}}^2\Big)^{\frac12} .
\end{equation*}
By monotonicity in $t$ and by taking the expectation, we obtain
\begin{equation*}
\begin{aligned}
\bE\Big[  \sup_{t\in[0,T]} K(\cdot,t)^2 
\Big]
&\lesssim
\bE\Big[ 
\int_0^T{\Norm{V}{ \int_0^r{ S_0(r-s) \Psi(s)\,\dd W(s)}}^2} \,\dd r \Big]
\\ &\quad +
\bE\Big[ 
 \sup_{t\in[0,T]}\Norm{H}{ \int_0^t { S_0(t-s) \Psi(s)\,\dd W(s)}}^2 \Big].
\end{aligned}
\end{equation*}
The proof of \eqref{boundstochintmeansq1} is now completed by the 
inequalities
\begin{align*}
\bE\Big[  \int_0^T{\Norm{V}{ \int_0^r{ S_0(r-s) \Psi(s)\,\dd W(s)}}^2} \,\dd r  
\Big] 
\leq \frac12 \bE\Big[\int_0^T { \norm{\sL_2(H)}{\Psi(s)Q^{\frac12}}^2 }\,\dd 
s\Big] 
\end{align*}
and 
\begin{align*}
\bE\Big[ \sup_{t\in[0,T]}  
\Norm{H}{ \int_0^t { S_0(t-s) \Psi(s)\,\dd W(s)}}^2 \Big] 
\leq 16 \bE\Big[ \int_0^T {
\norm{\sL_2(H)}{\Psi(s)Q^{\frac12}}^2 }\,\dd
s\Big] . 
\end{align*}
These are proved in \cite[Chapt.~3, Lemma 5.2]{Chow}. 
We sketch the proof of the second inequality; the first one
is proved by an eigenbasis expansion and Parseval's identity and can
be found in the cited reference. We 
introduce the notation
\begin{align*}
v(t) := \int_0^t{S_0(t-s)\Psi(s)\,\dd W(s)},  \quad 
z(t) := \int_0^t{\Psi(s)\,\dd W(s)}
\end{align*}
and integrate by parts, using $\frac{\partial}{\partial
  s}S_0(t-s)=A_0S_0(t-s)$ and  $\dd z=\Psi\,\dd W$, to get 
\begin{align*} 
v(t) = z(t) - A_0\int_0^t{S_0(t-s)z(s)}\,\dd s.
\end{align*}
By means of an eigenbasis expansion and Parseval's identity, we have 
\begin{align*}
\Norm{H}{ A_0\int_0^t{S_0(t-s)z(s)}\,\dd s}^2 \leq \sup_{s \in 
[0,t]}\norm{H}{z(s)}^2.
\end{align*}
Since $z$ is a martingale we can apply Doob's inequality followed by Ito's 
isometry:
\begin{align*}
\bE\Big[ \sup_{t\in[0,T]} \norm{H}{v(t)}^2 \Big] 
&\leq 4 \bE\Big[ \sup_{t\in[0,T]} 
\norm{H}{z(t)}^2 \Big] \leq 16 \bE\Big[ \norm{H}{z(T)}^2 \Big] 
\\
&=  16 \bE\Big[ 
\int_0^T {\norm{\sL_2(H)}{\Psi(s)Q^{\frac12}}^2 }\,\dd s\Big] . 
\end{align*}
This completes the proof.
\end{proof}

\begin{remark}
  The above proof relies heavily on the use of an eigenbasis
  expansion.  See \cite[Proposition 7.3]{DapZab} for a slightly weaker
  result for a more general semigroup.
\end{remark}


By means of the results presented above we have a unique solution
$ U_{\omega,t} = \big( U_1(\cdot),U_2 \big)_{\omega,t}$
of~\eqref{secondspacetimestoch} for every $t\in[0,T]$.  By uniqueness,
we have that $(U_1(s))_{\omega,t}=\big(U_1(s)\big)_{\omega,T}$ for
almost every $s\in[0,t]$.  This justifies the notation
\begin{align} \label{eq:form}
U(\omega,t) = \big( U_1(\omega,t) ,U_2(\omega,t) \big), 
\end{align}
where $U_1(\omega,t)=\big(U_1(\omega,t)\big)_{\omega,T}$ and
$U_2(\omega,t)=\big(U_2\big)_{\omega,t}$.  Since now $U_1 \notin \cX$,
we cannot conclude that $U_1$ is continuous and
$U_1 = U_2$. However, the following lemma ensures that $U_2$ is a
continuous version of $U_1$ and that, in particular, our concept of solution is
consistent with the variational solution of the same equation.

\begin{lemma}\label{lemma:versions}  Let $U=(U_1,U_2)$ of the
  form~\eqref{eq:form} be the unique solution
  of~\eqref{secondspacetimestoch}.  If we denote by $\cU$ the
  variational solution to \eqref{stocHeat} as in
  Definition~\ref{varsol}, then the following identities hold:
\begin{itemize}
\item[] $ U_1 = \bar{\cU}$ in $L^2((0,T);V)$, $\bP$-a.s.
\item[] $ U_2 = \cU$ in $H$, $\bP \otimes \dd t$-a.s.
\item[] $ U_1 = U_2 $ in $L^2((0,T);H)$, $\bP$-a.s.
\end{itemize}
Moreover,  
\begin{align*}
U_2(t) = U_0 + \int_0^t{\big(- A(s) U_1(s)+f(s)\big) }\,\dd s 
+  \int_0^t{\Psi(s)\, \dd W(s)}, 
\quad \text{$\bP$-a.s.} 
\end{align*}
In particular, it follows that $U_2$ is an $H$-valued continuous version of 
$U_1$.
\end{lemma}

\begin{proof}
For any $t\in[0,T]$ the variational solution is such that:
\begin{align*}
\cU(t) = U_0 + \int_0^t{\big(- A(s) \bar{\cU}(s)+f(s)\big) }\,\dd s 
+  \int_0^t{\Psi(s)\, \dd W(s)}, 
\quad \text{$\bP$-a.s.} 
\end{align*}
We  multiply this by arbitrary $\xi \in V$: 
\begin{align*}
\inner{\cU(t)}{\xi} 
= \inner{U_0}{\xi} + \int_0^t\inner{- A(s) \bar{\cU}(s)+f(s) }{\xi}\,\dd s  
+  
\int_0^t\inner{\Psi(s)\, \dd W(s) }{\xi}.  
\end{align*} 
By using Ito's formula (similarly to \cite[Lemma 5.5]{DaPrato}) on the
process  $\inner{\cU(t)}{\xi}\phi(t)$, where $\phi \in
H^1((0,T);\bR)$, we obtain 
\begin{align*}
\inner{\cU(t)}{\xi}\phi(t) 
&
=  \inner{U_0}{\xi}\phi(0) 
\\ & \quad 
+\int_0^t \Big(
\inner{ \bar{\cU}(s)}{\xi}\dot{\phi}(s) 
+\dual{-A(s) \bar{\cU}(s) +f(s)}{\xi}\phi(s) \Big) 
\,\dd s 
\\ &\quad 
+\int_0^t{\phi(s)\inner{\Psi(s) \,\dd W(s) }{\xi}}, 
\quad \text{$\bP$-a.s.} 
\end{align*}
This is the same as 
\begin{align*}
\inner{\cU(t)}{\phi(t)  \xi} 
&
=  \inner{U_0}{\phi(0) \xi}
+\int_0^t{ 
\tdual{ \bar{\cU}(s)}{  \dot{\phi}(s) \xi
-A^*(s) (\phi(s)\xi)}}
 \,\dd s 
\\ & \quad 
+\int_0^t{ 
\dual{f(s)}{\phi(s)\xi} }
\,\dd s 
+\int_0^t{\inner{\Psi(s)\, \dd W(s) }{\phi(s)\xi}}.   
\end{align*}
Since functions of the form $x = \phi  \xi$ are dense in
$\cX_0^t$, we conclude that, for almost all
$(\omega,t)\in\Omega \times 
[0,T]$, 
\begin{align*}
&\int_0^t
\tdual{ \bar{\cU}(s)} {-\dot{x}(s)+A^*(s) x(s)}
 \,\dd s 
+\inner{\cU(t)}{x(t)}
\\ &\quad =  \inner{U_0}{x(0)}
+\int_0^t 
\dual{f(s)} {x(s)}
\,\dd s 
+\int_0^t\inner{\Psi(s) \,\dd W(s) }{x(s)}, 
\quad \forall x\in\cX_0^t. 
\end{align*}
This means that $(\bar{\cU},\cU(t))\in \cY_0^t \times H$ is a solution
to~\eqref{secondspacetimestoch}.   The conclusions of the lemma now
follow by uniqueness of such a solution.  
\end{proof}


\begin{theorem}[Existence and uniqueness] \label{maintheorem} If
  $U_0 \in L^2(\Omega;H)$, $f \in L^2(\Omega \times (0,T) ;V^*)$ and
  $\Psi{Q}^{\frac{1}{2}} \in L^2(\Omega \times (0,T) ;\sL_2(H))$, then
  there exists a unique solution
  $ U = (U_1,U_2)\in L^2(\Omega \times (0,T) ;V) \times
  L^2(\Omega;\sC([0,T];H)) $
  of the form~\eqref{eq:form} to \eqref{secondspacetimestoch}. Its
  norm satisfies the bound
\begin{equation*}
\begin{aligned}
&\bE\Big[  \int_0^T{\norm{V}{U_1(t)}^2}\,\dd t +
\sup_{t\in[0,T]}\norm{H}{U_2(t)}^2 \Big] \\
&\qquad \qquad \lesssim c_B^{-1} \bE\Big[
\int_0^T{\norm{V^*}{f(t)}^2}\,\dd t +
\norm{H}{U_0}^2 
+
\int_0^T { \norm{\sL_2(H)}{\Psi(t)Q^{\frac12}}^2 }\,\dd t\Big]. 
\end{aligned}
\end{equation*}
where the constant hidden in $\lesssim$ only depends only on numerical factors.
\end{theorem}
\begin{proof}
  In view of the $\omega$-wise invertibility of the operator
  $B_{\omega,t}$, and the bounds for $\sF_{\omega,t}$
  in~\eqref{eq:boundF} and $\sW_{\omega,t}$
  in~\eqref{boundstochintmeansq2},  we have that for fixed $\omega$ and for any
$t \in [0,T]$, there exists a unique solution to 
\eqref{secondspacetimestoch}, which satisfies the bound
\begin{equation*}
\begin{aligned}
&
\int_0^t{\norm{V}{U_1(\omega,s)}^2}\,\dd s 
+
\norm{H}{U_2(\omega,t)}^2 
\lesssim
c_B^{-1}
\Big(\norm{(\cX_0^t)^*}{\sF_{\omega,t}}^2 
+
\norm{(\cX_0^t)^*}{\sW_{\omega,t}}^2 \Big)
\\
&\qquad \lesssim c_B^{-1}\Big(\norm{L^2((0,t);V^*)}{f(\omega,\cdot)}^2 +
\norm{H}{U_0(\omega)}^2 +K(\omega,t)^2 \Big).
\end{aligned}
\end{equation*}
In view of \eqref{boundrhsmeansq} and \eqref{boundstochintmeansq2},
this leads to 
\begin{equation*}
\begin{aligned}
&\bE\Big[  \int_0^T{\norm{V}{U_1(t)}^2}\,\dd t +
\sup_{t\in[0,T]}\norm{H}{U_2(t)}^2 \Big] 
\\ & \qquad 
\lesssim 
c_B^{-1}\bE\Big[ \sup_{t\in[0,T]}
\norm{(\cX_0^t)^*}{{\sF}_{\cdot,t}}^2 
+
\sup_{t\in[0,T]} \norm{(\cX_0^t)^*}{{\sW}_{\cdot,t}}^2 \Big]
\\
&\qquad \lesssim
c_B^{-1}\bE\Big[ \norm{L^2((0,T);V^*)}{f}^2 + \norm{H}{U_0}^2 +
\int_0^T { \norm{\sL_2(H)}{\Psi(t)Q^{\frac12}}^2 }\,\dd t
\Big].
\end{aligned}
\end{equation*}
Together with Lemma~\ref{lemma:versions} this concludes the proof of
the theorem.
\end{proof}

In the remainder of the manuscript we will sometimes use 
the alternative notation $ U \in L^2(\Omega \times (0,T) ;V) \cap 
L^2(\Omega;\sC([0,T];H)) $, equivalent to $ U = (U_1,U_2)\in 
L^2(\Omega \times (0,T) ;V) \times L^2(\Omega;\sC([0,T];H)) $, where the two 
components of $U$ are now understood as versions of the same object. 


\subsection{Connection with the mild solution}
We have already shown that a weak space-time solution is a variational 
solution. If we assume that $-A$ is independent of $\omega$ and 
$t$ and hence
generates an analytic semigroup $(S(t))_{t\geq 0} $, we can also show that a 
weak space-time solution is a mild solution. The following theorem
holds:
\begin{theorem}
Let $U$ be the mild solution \eqref{mild} to the problem \eqref{stocHeat} and
assume that $(U_1,U_2(t)) \in \cY_0^t \times H$ is the weak space-time
solution to the same problem, i.e., the solution to 
\eqref{secondspacetimestoch}. Then, for any $t \in [0,T]$, $ \, U_1
\overset{\cY}{=} U$ and $U_2(t) \overset{H}{=} U(t)$.
\end{theorem}
\begin{proof}
For any $t\in[0,T]$ and for any $x\in 
\cX_0^t$, we have $\bP$-a.s. that
\begin{equation}\label{theoremequation}
\begin{aligned}
& \int_0^t \tdual{U_1(s)}{-\dot{x}(s) +A^*x(s)}\,\dd s +
\inner{U_2(t)}{x(t)} \\
&\qquad = \int_0^t {\dual{f(s)}{x(s)}}\,\dd s + \inner{U_0}{x(0)} +
\int_0^t{\inner{\Psi(s)\,\dd W(s)}{x(s)}}.
\end{aligned}
\end{equation}
We now choose test functions $x=v$, where $v \in \cX^t_0$ is the
solution to the deterministic backward equation
\eqref{adjequationstrongform} over the time interval $[0,t]$, with
arbitrary final data $\xi \in H$ and load function
$g \in L^2((0,t);V^*)$. Its variational formulation is given by
\eqref{firstbackward}, that is,
\begin{align}\label{theoremvarform}
&
\int_0^t{\tdual{y_1(s)}{-\dot{v}(s) + A^*v(s)}}\,\dd s 
+ \inner{y_2}{v(t)} 
\\ & \qquad
= \int_0^t{\tdual{y_1(s)}{g(s)}}\,\dd s + \inner{y_2}{\xi},
\end{align}
for all $y\in \cY_0^t \times H$. The solution is given by the mild
solution formula
\begin{equation}\label{mildformulatheorem}
v(s) = S^*(t-s)\xi + \int_s^t{S^*(r-s)g(r)\,\dd r},\quad s\in[0,t],
\end{equation}
where $S^*$ is the semigroup generated by $-A^*$, namely $S^*(s) = \ee^{-sA^*}$.
By substituting $x= v$ in \eqref{theoremequation} and $y=(U_1,U_2(t))$
in \eqref{theoremvarform}, we obtain
\begin{align*}
& \int_0^t{ \tdual{U_1(s)}{g(s)}} \,\dd s +
\inner{U_2(t)}{\xi} \\
&\quad= \int_0^t{ \dual{f(s)}{v(s)}} \,\dd s +
\inner{U_0}{v(0)} + \int_0^t{\inner{\Psi(s)\,\dd W(s)}{v(s)}},\\
\intertext{which, by \eqref{mildformulatheorem}, in its turn is equal to}
&\quad = \int_0^t{\dual{f(s)}{ S^*(t-s)\xi}}\,\dd s + \int_0^t
{\bigdual{f(s)}{\int_s^t {S^*(r-s)g(r)}\,\dd r}}\,\dd s \\
&\quad\quad + \inner{U_0}{S^*(t)\xi} + \biginner{U_0}{\int_s^t 
{S^*(r-s)g(r)}\,\dd
r}  \\ 
&\quad\quad +\int_0^t{\inner{\Psi(s)\,\dd W(s)}{S^*(t-s)\xi}}+ 
\int_0^t{\biginner{\Psi(s)\,\dd
W(s)}{\int_s^t {S^*(r-s)g(r)}\,\dd r}} .
\end{align*}
By manipulating the dual pairings in a suitable way, changing the
order of integration (using the stochastic version of Fubini's
theorem), and using the mild solution formula \eqref{mild}, we get
\begin{align*}
& \int_0^t {\tdual{U_1(s)}{ g(s)}}\,\dd s + \inner{U_2(t)}{
\xi} \\
&\quad = \biginner{S(t)U_0 + \int_0^t{S(t-s)f(s)}\,\dd s +
\int_0^t{S(t-s)}\Psi(s)\,\dd W(s)}{\xi} \\
&\quad\quad + \int_0^t {\bigtdual{S(s)U_0 + \int_0^s {S(s-r)f(r)}\,\dd r +
\int_0^s{S(s-r)\Psi(r)\,\dd W(r) }}{ g(s)}\,\dd s  } \\
&\quad = \inner{U(t)}{ \xi} + \int_0^t{\tdual{U(s)}{g(s)}}\,\dd s ,
\end{align*}
which reads
\begin{equation*}
\duals{U_1 - U}{g}{\cY_0^t}{(\cY_0^t)^*} + \inners{U_2(t) - U(t)}{ \xi}{H} =
0.
\end{equation*}
Since $(g,\xi)$ is arbitrary in $L^2((0,t);V^*) \times H$, and $t\in[0,T]$, it
follows that
\begin{align*}
U_1 \overset{\cY}{=} U, \ U_2(t) \overset{H}{=} U(t), 
\quad t \in
[0,T],\, \bP\mbox{{\rm -a.s}}.
\end{align*}
\end{proof}
\begin{remark}\label{remarksecondcomponent}
This is consistent with the fact that $U_1$ is a $V$-valued version of $U_2$
and that $U_2$ is a continuous $H$-valued function of time.
\end{remark}

\section{Regularity}\label{Regularity}
In this section we briefly investigate the regularity properties of
the weak space-time solution. In order to 
simplify the presentation, we
assume now that $A$ is independent of $\omega$ and $t$ and
self-adjoint in addition to \eqref{eq:bounds_on_A}.  

Then $-A$ is the generator of an analytic semigroup $S(t)=\ee^{-tA}$
and fractional powers $A^s$, $s\in\mathbb{R}$, of $A$ are well
defined.  We define norms of fractional order
$\norm{\dot{H}^{s}}{v}:=\norm{H}{A^{\frac s2}v}$ for $s\in\mathbb{R}$.
For $s\ge0$ we define the spaces $\dot{H}^s=D(A^{\frac s2})$ and for
$s\le 0$ we define $\dot{H}^s$ to be the closure of $H$ with respect
to the $\dot{H}^s$-norm.  These spaces are Hilbert spaces, in
particular, $\dot{H}^0=H$, $\dot{H}^1\simeq V$, and
$\dot{H}^{-s}=(\dot{H}^s)^*$.

For $\beta\ge0$, we then consider the spaces
\begin{align*}
\cYb &:= L^2((0,t);\dot{H}^{1+\beta}),\\
\cXb &:= L^2((0,t);\dot{H}^{1-\beta}) \cap H^1((0,t);\dot{H}^{-1-\beta}),
\end{align*}
normed by
\begin{align*}
\norm{\cYb}{y}^2 &:= \int_0^t{\norm{\dot{H}^{1+\beta}}{y(s)}^2}\,\dd s,\\
\norm{\cXb}{x}^2 &:= \int_0^t{\big(\norm{\dot{H}^{1-\beta}}{x(s)}^2 +
\norm{\dot{H}^{-1-\beta}}{\dot{x}(s)}^2\big)}\,\dd s + 
\norm{\dot{H}^{-\beta}}{x(0)}^2 + \norm{\dot{H}^{-\beta}}{x(t)}^2.
\end{align*}
The spaces in the previous sections correspond to $\beta =0$. In particular,
as before, we use the notation $\cYbT = \cY_0^{T,\beta}$ and $\cXbT =
\cX_0^{T,\beta}$. The space $ \cYb
\times \Hb$ endowed with its product norm $\norm{\cYb \times \Hb}{\cdot}$ and
the space $ \cXb$ endowed with the norm $\norm{\cXb}{\cdot}$ are Hilbert 
spaces.

There is a dense embedding $\cXbT \hookrightarrow
\sC([0,T];\dot{H}^{-\beta})$, i.e.,  for
any $ {x \in \cXbT }$,
$$\norm{\sC([0,T];\dot{H}^{-\beta})}{x} \leq
\norm{\cXbT}{x},$$ where the embedding constant is the same as in 
\eqref{eq:embedding_constant}. A proof of this fact can be found in 
\cite{DauLio,Lunardi}, and
relies on the properties of the interpolation space
$$(\dot{H}^{1-\beta},\dot{H}^{-1-\beta} )_{\frac{1}{2}} = \dot{H}^{-\beta}.$$

We introduce a new bilinear
form, $\sB_{t,\beta}^*$, given by the original one, $\sB_{t}^*$ with
constant operator $A$,
restricted to the newly introduced spaces, that is,
\begin{align*}
&\sB_{t,\beta}^* \colon (\cYb \times \Hb) \times \cXb
\rightarrow \bR,
\end{align*}
together with new load functionals,
\begin{align*}
\sF_{\omega,t,\beta} \colon \cXb \rightarrow \bR,  \quad 
\sW_{\omega,t,\beta} \colon \cXb \rightarrow \bR,
\end{align*}
given by $\sF_{\omega,t}$ and ${\sW}_{\omega,t}$ defined on the new
spaces introduced above.

The weak space-time formulation reads, for almost every $(\omega,t) \in 
\Omega\times [0,T]$:
\begin{equation} \label{secondspacetimestochreg}
\begin{aligned}
&U_{\beta}(\omega,t) \in \cYb \times \Hb \colon\\
&\sB_{t,\beta}^*(U_{\beta}(\omega,t),x) =\sF_{\omega,t,\beta}(x) + 
\sW_{\omega,t,\beta}(x), \quad \forall x \in \cXb.
\end{aligned}
\end{equation}
It is possible to prove that the conditions \eqref{BDD}, \eqref{BNB1A} and
\eqref{BNB2A} still hold, with the same constants $C_B$ and $c_B$ as before. The
proof of this follows from a straightforward modification of the proof for the
deterministic framework in \cite{SchwabSte} or \cite{Fra}, taking in account the
remarks made for its extension to the stochastic framework in Section
\ref{varformstoch}. It will therefore be omitted.

In the following lemma we give sufficient conditions on the load
functionals in order to have a unique solution. 

\begin{lemma} With the notation introduced above, the following facts hold true:
\begin{itemize}
 \item If $f \in\cY_0^{t,\beta-2}$ and $U_0 \in \Hb$, $\bP$-a.s., then
$\sF_{\cdot,t,\beta} \in
(\cXb)^*$, $\bP$-a.s.  
Moreover, if $f \in L^2(\Omega;\cY^{\beta-2})$ and
$U_0 \in L^2(\Omega;\Hb)$, then
\begin{align*}
\bE\Big[ \sup_{t\in[0,T]} \norm{ (\cXb)^*}{\sF_{\cdot,t,\beta}}^2 \Big] 
\lesssim
\bE\Big[
\norm{\cY^{\beta-2}_H}{f}^2 + \norm{\Hb}{U_0}^2 \Big].
\end{align*}

\item If $\Psi Q^{\frac{1}{2}} \in L^2( \Omega \times 
(0,T); \sL_2(H,\dot{H}^{\beta})$, then
${\sW}_{\cdot,t,\beta} \in (\cXb)^*$, $\bP$-a.s.  
Moreover, 
\begin{align*}
\bE\Big[ \sup_{t\in[0,T]}  \norm{ (\cXb)^*}{{\sW}_{\cdot,t,\beta}}^2
\Big] \lesssim \bE\Big[ 
\int_0^T{ \norm{\sL_2(H,\dot{H}^{\beta})}{\Psi(t)Q^{\frac12}}^2  }\,\dd t \Big].
\end{align*}
 \end{itemize}
\end{lemma}

\begin{proof}
The first statement is obvious. In order to prove the second one, one can use
the same
notation and techniques as in Section~\ref{varformstoch}, together with the
employment of the following inequalities to derive an analogue of Lemma
\ref{lemmastochint}:
\begin{align*}
\bE\Big[ \sup_{t\in[0,T]}  \Norm{H}{ A^{\frac{\beta}{2}}\int_0^t { 
S(t-s)\Psi(s)\,\dd
W(s)}}^2 \Big] \lesssim \bE\Big[  \int_0^T {
\norm{\sL_2(H)}{A^{\frac{\beta}{2}} \Psi(t) Q^{\frac12}}^2 }\,\dd t \Big]
\end{align*}
and 
\begin{align*}
\bE\Big[  \int_0^T{\Norm{H}{ A^{\frac{\beta-1}{2}} \int_0^r{ S(r-s)\Psi(s)\,\dd
W(s)}}^2} \,\dd
r  \Big] \lesssim \bE\Big[  \int_0^T { \norm{\sL_2(H)}{A^{\frac{\beta}{2}} 
\Psi(t) Q^{\frac12}}^2
}\,\dd t \Big].
\end{align*}
These two properties are direct generalizations of the ones presented
in Lemma \ref{lemmastochint}.
\end{proof}
The previous lemma, together with the initial remarks about the fulfilment of
the conditions \eqref{BDD}, \eqref{BNB1A}, and \eqref{BNB2A}, gives the
following result.
\begin{theorem}\label{regulartheorem}
  Let $\beta \geq 0$ and $f \in L^2(\Omega \times (0,T) ; 
\dot{H}^{\beta-1})$, $U_0 \in
  L^2(\Omega;\Hb)$, and $\Psi Q^{\frac{1}{2}} \in
  L^2(\Omega \times (0,T) ;\sL_2(H,\dot{H}^{\beta}))$. Then the problem
  \eqref{secondspacetimestochreg} has a unique solution $ U \in
  L^2(\Omega \times (0,T); \dot{H}^{\beta+1}) \cap L^2(\Omega;\sC([0,T];\Hb)) $ 
and its norm is
  bounded by
\begin{equation*}
\begin{aligned}
&\bE\Big[  \int_0^T{\norm{\dot{H}^{\beta+1}}{U_1(t)}^2}\,\dd t +
\sup_{t\in[0,T]}\norm{\dot{H}^{\beta}}{U_2(t)}^2 \Big] \\
&\qquad \qquad \lesssim c_B^{-1}\bE\Big[
\int_0^T{\norm{\dot{H}^{\beta-1}}{f(t)}^2}\,\dd t +
\int_0^T{\norm{\sL_2(H,\Hb)}{\Psi(t)Q^{\frac12}}^2}\,\dd t +
\norm{\dot{H}^{\beta}}{U_0}^2 \Big],
\end{aligned}
\end{equation*}
where the constant hidden in $\lesssim$ depends only on numerical 
factors.
\end{theorem}


\section{Linear multiplicative noise}\label{linmultnoise}
In this section we use the theory developed in the previous sections to prove
existence and uniqueness to the weak space-time solution to the problem
\begin{equation}\label{multiplicativestrong}
\begin{aligned}
&\dd U(t) + A(t)U(t)\,\dd t = f(t)\,\dd t + (G(t)U(t))\,\dd W(t), \quad t \in
(0,T], \\
&U(0) = U_0. &&
\end{aligned}
\end{equation}
Here $G(\omega,t) \in \sL(H,\sL(H))$, with further assumptions on its
$(\omega,t)$-dependence to be specified below.
As we have done before, we introduce an $\omega$-wise weak formulation.
In order to do so we introduce a
new load functional $\sW^{v}_{\omega,t}$ defined by 
\begin{align*}
\sW_{\omega,t}^{v} \colon \cX_0^t \rightarrow \bR,\quad 
{\sW}^{v}_{\omega,t}(x) =
\Big(\int_0^t{\inner{(G(s)v(s))\,\dd W(s)}{x(s)}}\Big)(\omega),
\end{align*}
for $(\omega,t)\in \Omega\times [0,T]$ and $v \in \sS_T:=
L^2(\Omega;L^2((0,T);V)) \cap L^2(\Omega;\sC([0,T];H))$.  The weak
space-time formulation of problem \eqref{multiplicativestrong} reads
hence, for almost every $(\omega,t)\in \Omega\times [0,T]$, 
\begin{equation} \label{secondspacetimestochnonlfix}
U_{\omega,t} \in \cY_0^t \times H : \sB_{\omega,t}^*(U_{\omega,t},x) =
\sF_{\omega,t}(x) +\sW^{U_{\omega,t}}_{\omega,t}(x)  ,
\quad \forall x \in \cX_0^t.
\end{equation}
We use Banach's fixed point theorem for the linear operator $\sT\colon
v\mapsto U$ that maps $v\in\sS_T$ to the solution of of the problem
\begin{equation} \label{secondspacetimestochnonlfixaux}
U_{\omega,t} \in \cY_0^t \times H : \sB_{\omega,t}^*(U_{\omega,t},x) =
\sF_{\omega,t}(x) +\sW^{v}_{\omega,t}(x),
\quad \forall x \in \cX_0^t. 
\end{equation}
We will show that $\sT\colon\sS_T\to\sS_T$ is a contraction, if $T$ is
small.
We introduce the notation $\sL_2^0(H)$ for the space of operators $\Psi$ such
that 
\begin{align*}  
\norm{\sL_2^0(H)}{\Psi} := \norm{\sL_2(H)}{\Psi Q^{\frac12}}<\infty.
\end{align*}
We make the further assumption that $G$ is predictable, 
bounded with respect to $\omega$, and $L^p$ in time for some
$p>2$, i.e., for some constant $\kappa$, 
\begin{equation}\label{LpboundonPsi} \operatorname*{ess\
sup}_{\omega\in\Omega} \Big( \int_0^T{
\norm{\sL(H,\sL_2^0(H))}{G(\omega,t)}^p }\,\dd t \Big)^{1/p} \leq
\kappa.
\end{equation}
An example is presented in Remark~\ref{remark5} below. 

\begin{lemma}\label{inequalityonB}
For any $v\in\sS_T$ and $G$ as in \eqref{LpboundonPsi}, it holds that
\begin{equation*}
\begin{aligned}
\bE\Big[
\int_0^T{ \norm{\sL_2(H)}{(G(t)\,v(t))Q^{\frac12}}^2 }\,\dd t
 \Big] \le 
T^{\frac{p}{p-2}}  \kappa^2 
\norm{\sS_T}{ v}^2.
\end{aligned}
\end{equation*}
\end{lemma}
\begin{proof} We use H\"older's inequality to get 
\begin{align*}
&\bE\Big[
\int_0^T{ \norm{\sL_2(H)}{(G(\cdot,t)\,v(\cdot,t))Q^{\frac12}}^2 }\,\dd t
 \Big] \leq \bE\Big[
\int_0^T{ \norm{\sL(H,\sL_2^0(H))}{G(\cdot,t)}^2 \norm{H}{v(\cdot,t)}^2 } \,\dd 
t
 \Big] \\
  &\qquad \leq \bE\Big[ \sup_{t\in[0,T]}{  \norm{H}{v(\cdot,t)}^2 }
\Big(\int_0^T{ \norm{\sL(H,\sL_2^0(H))}{G(\cdot,t)}^2 } \,\dd t \Big)
 \Big] \\
 &\qquad \leq \bE\Big[ \sup_{t\in[0,T]}{  \norm{H}{v(\cdot,t)}^2 }\Big]\,
T^{\frac{p}{p-2}}
\operatorname*{ess\ sup}_{\omega\in\Omega}
\Big(\int_0^T{ \norm{\sL(H,\sL_2^0(H))}{G(\omega,t)}^{p} } \,\dd t
\Big)^{\frac{2}{p}}
 \\
 &\qquad \leq T^{\frac{p}{p-2}} \kappa^2 \, \norm{\sS_T}{ v}^2,
\end{align*}
where in the last line we used \eqref{LpboundonPsi}. 
\end{proof}

By combining Lemmas \ref{inequalityonB} and \ref{lemmastochint}, with
$\Psi=Gv$, we see that ${\sW}^{v}_{\omega,t}\in (\cX_0^t)^*$ and
\begin{equation*}
\bE\Big[ 
\sup_{t\in[0,T]} \norm{(\cX_0^t)^*}{{\sW}^{v}_{\omega,t}}^2
\Big]
\lesssim 
\bE\Big[ \int_0^T{\norm{\sL_2(H)}{(G(t)v(t))Q^{\frac12}}^2}\,\dd t\Big] 
\lesssim 
T^{\frac{p}{p-2}}\,\norm{\sS_T}{ v}^2.  
\end{equation*}
If $U_0 \in L^2(\Omega;H)$, $f \in L^2(\Omega \times (0,T) ;V^*)$,
$Q^{\frac{1}{2}} \in \sL_2(H)$, then we may refer to
Theorem~\ref{maintheorem} to conclude that
\eqref{secondspacetimestochnonlfixaux} has a unique solution with
\begin{equation*}
\bE\Big[  \int_0^T{\norm{V}{U_1}^2}\,\dd t 
+
\sup_{t\in[0,T]}\norm{H}{U_2}^2 \Big] 
\lesssim \bE\Big[
\int_0^T{\norm{V^*}{f}^2}\,\dd t 
+
\norm{H}{U_0}^2 \Big]
+
T^{\frac{p}{p-2}}\,\norm{\sS_T}{ v}^2.  
\end{equation*}
Hence, the solution operator $\sT$ maps $\sS_T$ to itself. An
application of the previous bound with $f=0$, $U_0=0$ shows that it is
a contraction, if $T$ is small.  We thus have a unique solution on some
short interval $[0,T_0]$ and, since the interval of existence does not
depend on the size of the data $f, U_0$, we may repeat the argument
and extend it to $[T_0,2T_0]$, $[2T_0,3T_0]$, and so on until we
obtain a solution on $[0,T]$. 

We summarize the result in the following theorem:

\begin{theorem}[Existence and uniqueness]
  If $U_0 \in L^2(\Omega;H)$, $f \in L^2(\Omega \times (0,T) ;V^*)$,
  and $ G \in L^{\infty}(\Omega;L^p((0,T);\sL(H,\sL_2^0(H))))$ for
  some $p>2$, see \eqref{LpboundonPsi}, then
\eqref{secondspacetimestochnonlfix} has a unique solution 
  $ U \in L^2(\Omega \times (0,T) ;V) \cap L^2(\Omega;\sC([0,T];H)) $.
\end{theorem}

\begin{remark}
  This approach extends easily to a semilinear equation of the form 
  \begin{align*}
    \dd U(t) + A(t)U(t)\,\dd t = F(t,U(t))\,\dd t + G(t,U(t))\,\dd W(t)
  \end{align*}
under appropriate global Lipschitz assumptions on the nonlinear
operators $F$, $G$. 
\end{remark}
\begin{remark}
Under the assumptions of Section \ref{Regularity} and,  for some $p>2$, 
\begin{equation*}
\operatorname*{ess\ sup}_{\omega\in\Omega}
\Big(\int_0^T{  \norm{\sL(\Hb,\sL_2^0(\Hb))}{G(\omega,t)}^p}
\Big)^{1/p}\,\dd t 
\leq \kappa,
\end{equation*}
we may extend the regularity result of Theorem \ref{regulartheorem} to
\eqref{secondspacetimestochnonlfix}.
\end{remark}

\begin{remark}   \label{remark5}
  We present an example of an operator satisfying \eqref{LpboundonPsi}.
  Let $H=L^2(D)$ and define, for all $v,w\in H$ and for some function
  $g$, 
  \begin{align*}
    \Big((G(\omega,t)v)w\Big)(\xi)=g(\omega,t,\xi)v(\xi)w(\xi) , \quad \xi\in 
D. 
  \end{align*}
  Let $\{e_j\}_{j=1}^\infty\subset H$ be an ON basis such that
  $\sup_{j\ge1}\norm{L^\infty(D)}{e_j}\le C$.  This can be
  achieved, for example, when $D$ is a parallellogram in $\bR^d$.
  Then construct $Qv=\sum_{j=1}^\infty \gamma_j \inner{v}{e_j}e_j$,
  where the eigenvalues $\{\gamma_j\}_{j=1}^\infty$ are chosen so that
  $\sum_{j=1}^\infty\gamma_j=\norm{\sL_2(H)}{Q^{\frac12}}^2<\infty$.
  Then
\begin{align*}
&  \norm{\sL_2^0(H)}{G(\omega,t)v}^2 
= \norm{\sL_2(H)}{(G(\omega,t)v)Q^{\frac12}}^2 
= \sum_{j=1}^\infty \norm{H}{(G(\omega,t)v)Q^{\frac12}e_j}^2 
\\ & \quad
= \sum_{j=1}^\infty\gamma_j \norm{H}{(G(\omega,t)v)e_j}^2 
= \sum_{j=1}^\infty\gamma_j \norm{L^2(D)}{g(\omega,t,\cdot)ve_j}^2 
\\ &  \quad
\le \sum_{j=1}^\infty\gamma_j \norm{L^\infty(D)}{g(\omega,t,\cdot)}^2 
\norm{L^2(D)}{v}^2 \norm{L^\infty(D)}{e_j}^2 
\lesssim 
     \norm{L^\infty(D)}{g(\omega,t,\cdot)}^2 \norm{L^2(D)}{v}^2, 
\end{align*}
Therefore, 
\begin{align*}
  \norm{\sL(H,\sL_2^0(H))}{G(\omega,t)} \lesssim 
  \norm{L^\infty(D)}{g(\omega,t,\cdot)} 
\end{align*}
and \eqref{LpboundonPsi} follows if we assume that $g\in
L^\infty(\Omega;L^p((0,T);L^\infty(D)))$. 
\end{remark}
\paragraph{Acknowledgement}  The authors would like to thank the
anonymous referee for the constructive criticism.

\bibliographystyle{siam}  
\bibliography{biblionew}

\def\cprime{$'$}
\begin{thebibliography}{10}

\bibitem{BabuskaAziz}
{\sc I.~Babu{\v{s}}ka and A.~K. Aziz}, {\em Survey lectures on the mathematical
  foundations of the finite element method}, in The {M}athematical
  {F}oundations of the {F}inite {E}lement {M}ethod with {A}pplications to
  {P}artial {D}ifferential {E}quations ({P}roc. {S}ympos., {U}niv. {M}aryland,
  {B}altimore, {M}d., 1972), Academic Press, New York, 1972, pp.~1--359.

\bibitem{BabuskaJanik}
{\sc I.~Babu{\v{s}}ka and T.~Janik}, {\em The {$h$}-{$p$} version of the finite
  element method for parabolic equations. {I}. {T}he {$p$}-version in time},
  Numer. Methods Partial Differential Equations, 5 (1989), pp.~363--399.

\bibitem{CheginiStev}
{\sc N.~Chegini and R.~Stevenson}, {\em Adaptive wavelet schemes for parabolic
  problems: sparse matrices and numerical results}, SIAM J. Numer. Anal., 49
  (2011), pp.~182--212.

\bibitem{Chow}
{\sc P.~Chow}, {\em Stochastic {P}artial {D}ifferential {E}quations}, Chapman
  \& Hall/CRC Applied Mathematics and Nonlinear Science Series, Chapman \&
  Hall/CRC, Boca Raton, FL, 2007.

\bibitem{DaPrato}
{\sc G.~Da~Prato}, {\em An {I}ntroduction to {Infinite-Dimensiona}l
  {A}nalysis}, Universitext, Springer-Verlag, Berlin, 2006.

\bibitem{DapZab}
{\sc G.~Da~Prato and J.~Zabczyk}, {\em Stochastic {Equations} in {Infinite}
  {Dimensions}}, vol.~44 of Encyclopedia of Mathematics and its Applications,
  Cambridge University Press, Cambridge, 1992.

\bibitem{DauLio}
{\sc R.~Dautray and J.~L. Lions}, {\em Mathematical {A}nalysis and {N}umerical
  {M}ethods for {S}cience and {T}echnology. {V}ol. 5}, Springer-Verlag, Berlin,
  1992.

\bibitem{Ern}
{\sc A.~Ern and J.~L. Guermond}, {\em Theory and {P}ractice of {F}inite
  {E}lements}, vol.~159 of Applied Mathematical Sciences, Springer-Verlag, New
  York, 2004.

\bibitem{Lunardi}
{\sc A.~Lunardi}, {\em Interpolation {T}heory}, Appunti. Scuola Normale
  Superiore di Pisa (Nuova Serie). [Lecture Notes. Scuola Normale Superiore di
  Pisa (New Series)], Edizioni della Normale, Pisa, second~ed., 2009.

\bibitem{PrevotRockner}
{\sc C.~Pr{\'e}v{\^o}t and M.~R{\"o}ckner}, {\em A {Concise} {Course} on
  {Stochastic} {Partial} {Differential} {Equations}}, vol.~1905 of Lecture
  Notes in Mathematics, Springer, Berlin, 2007.

\bibitem{SchwabSte}
{\sc C.~Schwab and R.~Stevenson}, {\em Space-time adaptive wavelet methods for
  parabolic evolution problems}, Math. Comp., 78 (2009), pp.~1293--1318.

\bibitem{SchwabSuli}
{\sc C.~Schwab and E.~S{\"u}li}, {\em Adaptive {G}alerkin approximation
  algorithms for {K}olmogorov equations in infinite dimensions}, Stoch. Partial
  Differ. Equ. Anal. Comput., 1 (2013), pp.~204--239.

\bibitem{Fra}
{\sc F.~Tantardini}, {\em Quasi {O}ptimality in the {B}ackward
  {E}uler-{G}alerkin {M}ethod for {L}inear {P}arabolic {P}roblems}.
\newblock Tesi di dottorato, Universita' degli Studi di Milano, 2013.

\bibitem{UrbanPatera}
{\sc K.~Urban and A.~T. Patera}, {\em A new error bound for reduced basis
  approximation of parabolic partial differential equations}, C. R. Math. Acad.
  Sci. Paris, 350 (2012), pp.~203--207.

\bibitem{Yan}
{\sc Y.~Yan}, {\em Semidiscrete {G}alerkin approximation for a linear
  stochastic parabolic partial differential equation driven by an additive
  noise}, BIT, 44 (2004), pp.~829--847.

\end{thebibliography}

\end{document}